\documentclass{amsart}
\usepackage{amsthm,amssymb,amsmath,dsfont,bigdelim}
\usepackage{graphicx}
\usepackage{float}
\usepackage[arc,all]{xy}
\usepackage{longtable}
\usepackage{libertinust1math}
\usepackage[T1]{fontenc}
\newtheorem{thm}{Theorem}[section]
\newtheorem{cor}[thm]{Corollary}
\newtheorem{lem}[thm]{Lemma}
\newtheorem{prop}[thm]{Proposition}
\theoremstyle{definition}

\newtheorem*{theorem}{\textbf{Main Theorem}}
\newtheorem*{proof M}{\textbf{Proof of  Main Theorem}}
\theoremstyle{remark}

\numberwithin{equation}{section}

\begin{document}
\title[ the Schur's theorem]{On converse of the  Schur's theorem for nilpotent Lie superalgebras}%
\author[A. Shamsaki]{afsaneh shamsaki}%
\email{Shamsaki.afsaneh@yahoo.com}%
\address{School of Mathematics and Computer Science,
Damghan University, Damghan, Iran}
\author[P. Niroomand]{Peyman Niroomand}\thanks{Corresponding Author: Peyman Niroomand}
\email{niroomand@du.ac.ir, p$\_$niroomand@yahoo.com}
\address{School of Mathematics and Computer Science,
Damghan University, Damghan, Iran}
\author[E. Stitzinger]{ERNEST STITZINGER}
\email{stitz@ncsu.edu}
\address{Department of Mathematics, North Carolina State University, North Carolina, USA}


\begin{abstract}

In this paper, we establish a converse to Schur's theorem for Lie superalgebras \( L \), focusing on cases where the minimal generator number pairs \((p \vert q)\) of \( L/Z(L) \) are considered, and where the superdimension \( \mathrm{sdim} L^{2} \) is finite. We introduce a new invariant \( st(L) \), which plays a key role in the classification of finite-dimensional nilpotent Lie superalgebras. Specifically, we classify the structure of all such Lie superalgebras \( L \) when \( st(L) \in \{(0,0), (1,0), (0,1), (2,0), (0,2), (1,1)\} \).

\end{abstract}
\maketitle
  \section{Introduction and preliminaries}
Schur's theorem, a cornerstone in group theory, asserts that if the quotient of a group \( G \) by its center \( Z(G) \), denoted \( G/Z(G) \), is finite, then the commutator subgroup \( G^2 \) must also be finite. This theorem, named after Issai Schur \cite{schur}, has deepened the understanding of the structure of groups by highlighting the connection between the group \( G \) and its commutator subgroup. This relationship is significant in both abstract group theory and the study of specific types of groups. However, the theorem is not biconditional; in other words, the finiteness of the commutator subgroup \( G^2 \) does not necessarily imply the finiteness of the quotient group \( G/Z(G) \). A notable counterexample is provided by infinite extra special \( p \)-groups for odd primes \( p \), in which \( G^2 \) is finite, but the quotient \( G/Z(G) \) is infinite.

Over the years, this one-sidedness has raised interest among group theorists and algebraists alike, leading to attempts at determining specific conditions under which a partial converse to Schur's theorem might hold. Such conditions would provide greater insight into the structural properties of certain classes of groups. One of the most influential results in this area was established by Neumann \cite{neu}, who formulated an inequality involving the generating number of a group \( G \). He demonstrated that if \( G \) is generated by \( d(G) \) elements, then the quotient group satisfies the inequality \( |G/Z(G)| \leq |G^2|^{d(G)} \). This result has profound implications, particularly because it reveals that if both \( G^2 \) and the generating number \( d(G) \) are finite, then \( G/Z(G) \) must also be finite. Thus, under these additional constraints, a partial converse to Schur’s theorem can be established.

Schur’s theorem is not confined to the realm of group theory alone. In fact, its principles have been extended to other algebraic structures, including Lie algebras. The analog of Schur’s theorem in Lie algebras has been the subject of extensive research, with important results appearing in several foundational works, such as \cite{Ar}. These results explore the connections between a Lie algebra and its derived algebra, revealing structural insights that parallel those found in the group-theoretic context. One key result from \cite{sh1} builds on work done for groups in \cite{gum} and establishes that for a Lie algebra \( L \), the dimension of the quotient \( L/Z(L) \) by its center is governed by an equation that relates it to the dimension of its derived algebra \( L^2 \). This equation takes the form:

\[
\dim L/Z(L) = d(L/Z(L)) \cdot \dim L^2 - t(L),
\]

where \( t(L) \geq 0 \) is an integer. This result provides a foundation for a possible converse of Schur’s theorem in the context of Lie algebras, where finiteness conditions on the derived algebra and the generating number \( d(L/Z(L)) \) lead to corresponding finiteness properties in the quotient \( L/Z(L) \).

Building on these developments, recent research has explored whether analogous results can be extended to Lie superalgebras, a generalization of Lie algebras that incorporates a \( \mathbb{Z}_2 \)-grading. Lie superalgebras arise naturally in various fields of mathematics and theoretical physics, particularly in the study of supersymmetry, and they introduce an additional layer of complexity to classical results from Lie algebra theory. As in the case of Lie algebras, one might ask whether Schur's theorem and its converse can be extended to Lie superalgebras, and under what conditions these extensions hold.

A Lie superalgebra is an algebraic structure that generalizes the notion of a Lie algebra by splitting the algebra into two components: an even part, denoted \( L_{\bar{0}} \), and an odd part, denoted \( L_{\bar{1}} \). If \( \dim L_{\bar{0}} = r \) and \( \dim L_{\bar{1}} = s \), then the superdimension of \( L \) is given by \( \mathrm{sdim} L = (r, s) \). The structure of Lie superalgebras presents new challenges when attempting to generalize results from Lie algebra theory, as the interaction between the even and odd components must be carefully considered. Nonetheless, many classical results can be successfully extended to this richer algebraic setting.

In this paper, we seek to extend the work of \cite{sh1} to the context of Lie superalgebras. Our goal is to establish a generalization of the converse of Schur’s theorem for Lie superalgebras, specifically in the setting where the superdimension of the derived algebra is finite. We consider a Lie superalgebra \( L \) with superdimension \( \mathrm{sdim} L^2 = (r \vert s) \), where \( r + s \) is finite, and we assume that the quotient \( L/Z(L) \) is a \( (p \vert q) \)-generated Lie superalgebra. Under these conditions, we prove the inequality:

\[
\mathrm{sdim} L/Z(L) \leq \uplambda(L^2, p, q),
\]

where \( \uplambda(L^2, p, q) \) is defined as:

\[
\uplambda(L^2, p, q) = \left(p \cdot \dim L^2_{\overline{0}} + q \cdot \dim L^2_{\overline{1}}, \, q \cdot \dim L^2_{\overline{0}} + p \cdot \dim L^2_{\overline{1}}\right).
\]

This inequality provides a generalization of the converse of Schur’s theorem to the setting of Lie superalgebras, mirroring the result found in \cite[Theorem 3.1]{Nay1} for Lie algebras. In proving this result, we further the understanding of the relationship between the structure of a Lie superalgebra and its commutator superalgebra, enriching both the theory of Lie superalgebras and their applications.

In addition to extending the converse of Schur’s theorem, we also undertake the task of classifying finite-dimensional nilpotent Lie superalgebras based on the values of a specific invariant denoted \( st(L) \). Specifically, we classify nilpotent Lie superalgebras for which \( st(L) \) takes values in the set \( \{(0,0), (1,0), (0,1), (2,0), (0,2), (1,1)\} \). This classification provides a comprehensive understanding of the possible algebraic structures that arise under these conditions, offering a deeper insight into the superdimension and derived structure of these superalgebras.

The main contributions of this paper are twofold: first, we provide a generalization of the converse of Schur’s theorem to Lie superalgebras, and second, we offer a classification of finite-dimensional nilpotent Lie superalgebras based on specific superdimension constraints. These results not only expand the scope of classical algebraic theorems but also open up new avenues for future research in the realm of Lie superalgebras and their applications to theoretical physics. We anticipate that the findings presented here will stimulate further study into the structural properties of superalgebras and their invariants, as well as their connections to related fields.

In preparation for the detailed proofs presented in the subsequent sections, we briefly review some key definitions and terminology that will be used throughout the paper. The upper central series of a Lie superalgebra \( L \) is defined by:

\[
\lbrace 0\rbrace =Z_0(L)\subseteq  Z_1(L) \subseteq \dots \subseteq Z_k(L) \subseteq \dots,
\]

where \( Z_1(L) = Z(L) \) and \( Z_i(L)/Z_{i-1}(L) = Z(L/Z_{i-1}(L)) \). If \( L \) is a nilpotent Lie superalgebra of class \( k \), then \( Z_k(L) = L \).

A Lie superalgebra is called a stem Lie superalgebra if its center \( Z(L) \) is contained within its derived superalgebra \( L^2 \). Throughout this paper, we use the classification of stem nilpotent Lie superalgebras from \cite{Al}, which details all such algebras with superdimension at most \( (k \vert l) \) where \( k + l \leq 5 \) and \( \mathrm{sdim} L^2 = (r \vert s) \), with \( r + s \geq 2 \), over a field of characteristic not equal to 2 or 3.

\begin{align*}
& (4\vert 0)_{2}: \quad [e_1, e_2]=e_3, [e_1, e_3]=e_4\cr
&(2\vert 2)_{1}:  \quad [f_1, f_1]=e_1, [f_2, f_2]=e_2\cr
&(2\vert 2)_{4}:  \quad [f_1, f_2]=e_1, [f_2, f_2]=e_2\cr
&(2\vert 2)_{6}:  \quad [e_2, f_2]=f_1, [f_2, f_2]=e_1\cr
&(1\vert 3)_{1}:  \quad [e_1, f_2]=f_1, [e_1, f_3]=f_2\cr
&(5\vert 0)_{3}:  \quad [e_1, e_2]=e_3, [e_1, e_3]=e_4, [e_1, e_4]=e_5, [e_2, e_3]=e_5
\cr
&(5\vert 0)_{4}:  \quad [e_1, e_2]=e_3, [e_1, e_3]=e_4, [e_1, e_4]=e_5
\cr
&(5\vert 0)_{5}:  \quad [e_1, e_2]=e_3, [e_1, e_4]=e_5, [e_2, e_3]=e_5
\cr
&(5\vert 0)_{6}:  \quad [e_1, e_2]=e_3, [e_1, e_3]=e_4, [e_2, e_3]=e_5
\cr
&(5\vert 0)_{8}:  \quad [e_1, e_2]=e_4, [e_1, e_3]=e_5
\cr
&(4\vert 1)_{4}:  \quad [e_1, e_2]=e_3, [f_1, f_1]=e_4
\cr
&(4\vert 1)_{6}:  \quad [e_1, e_2]=e_3, [e_1, e_3]=e_4, [f_1, f_1]=e_4
\cr
&(1\vert 4)_{7}:  \quad [e_1, f_2]=f_1, [e_1, f_3]=f_2, [e_1, f_4]=f_3
\cr
&(1\vert 4)_{8}:  \quad [e_1, f_2]=f_1,  [e_1, f_4]=f_3
\cr
&(3\vert 2)_{5}:  \quad [f_1, f_1]=e_2,  [f_1, f_2]=e_1, [f_2, f_2]=e_3
\cr
&(3\vert 2)_{11}:  \quad [e_1, e_2]=e_3,  [e_1, f_2]=f_1
\cr
&(3\vert 2)_{12}:  \quad [e_1, e_2]=e_3,  [e_1, f_2]=f_1, [f_2, f_2]=e_3
\cr
&(3\vert 2)_{13}:  \quad [e_1, e_2]=e_3,  [e_1, f_2]=f_1, [f_1, f_2]=e_3, [f_2, f_2]=2e_2
\cr
&(2\vert 3)_{5}:  \quad  [f_1, f_1]=e_1, [f_2, f_2]=e_2, [f_3, f_3]=e_1
\cr
&(2\vert 3)_{6}:  \quad  [f_1, f_1]=e_1, [f_2, f_2]=e_2, [f_3, f_3]=e_1+e_2
\cr
&(2\vert 3)_{8}:  \quad  [f_1, f_2]=e_1, [f_2, f_2]=2e_2, [f_2, f_3]=e_2
\cr
&(2\vert 3)_{9}:  \quad  [f_1, f_2]=e_1, [f_2, f_2]=2e_2, [f_3, f_3]=e_1
\cr
&(2\vert 3)_{10}:  \quad  [f_1, f_2]=e_1, [f_2, f_2]=2e_2, [f_3, f_3]=e_1+e_2
\cr
&(2\vert 3)_{11}:  \quad  [f_1, f_2]=e_1, [f_2, f_2]=2e_2, [f_2, f_3]=e_2, [f_3,f_3]=e_1
\cr
&(2\vert 3)_{13}:  \quad  [e_1, f_3]=f_1, [f_2, f_2]=e_2
\cr
&(2\vert 3)_{14}:  \quad  [e_1, f_3]=f_1, [f_2, f_3]=e_2
\cr
&(2\vert 3)_{16}:  \quad  [e_1, f_3]=f_1, [f_2, f_2]=e_2, [f_3, f_3]=e_2
\cr
&(2\vert 3)_{18}:  \quad  [e_1, f_3]=f_1, [e_2, f_2]=f_1, [f_2, f_2]=2e_1, [f_2, f_3]=-e_2
\cr
&(2\vert 3)_{19}:  \quad  [e_1, f_3]=f_1, [e_2, f_2]=f_1, [f_2, f_3]=-e_1, [f_3, f_3]=2e_2
\cr
&(2\vert 3)_{20}:  \quad  [e_1, f_3]=f_1, [e_2, f_3]=f_2
\cr
&(2\vert 3)_{22}:  \quad  [e_1, f_2]=f_1, [e_1, f_3]=f_2
\cr
&(2\vert 3)_{23}:  \quad  [e_1, f_2]=f_1, [e_1, f_3]=f_2,   [f_1,f_3]=-e_2, [f_2,f_2]=e_2
\cr
&(2\vert 3)_{24}:  \quad  [e_1, f_2]=f_1, [e_1, f_3]=f_2,   [e_2,f_3]=f_1
\end{align*}
In the table below, we present the superdimension $ \mathrm{sdim} L/Z(L) $, $ \mathrm{sd}( L/Z(L) ) $, and $ \mathrm{sdim} L^{2} $ for all stem nilpotent Lie superalgebras $ L $ with superdimension $ (k\vert l) $, where $ k + l \leq 5 $ and $ \mathrm{sdim} L^{2} = (r\vert s) $ with $ r + s \geq 2. $
\begin{longtable}{ccccc}
\multicolumn{3}{c}{\textbf{Table 1. } }\\
\hline \multicolumn{1}{c}{\textbf{Name}} & \multicolumn{1}{c}{$ \mathrm{sdim}L/Z(L) $}   & \multicolumn{1}{c}{$\mathrm{sd}( L/Z(L) )$} & \multicolumn{1}{c}{$ \mathrm{sdim} L^{2} $} \\
\hline
\endhead
\hline \multicolumn{2}{r}{\small \itshape continued on the next page}
\endfoot
\endlastfoot
$(4\vert 0)_{2}$  &$ (3,0) $& $ (2,0) $ & $ (2,0) $\\
\\
$(2\vert 2)_{1},$ $(2\vert 2)_{4}$ & $ (0,2) $& $ (0,2) $ & $ (2,0) $\\
\\
$(2\vert 2)_{6}$& $ (1,1) $& $ (1,1) $ & $ (1,1) $\\
\\
$(1\vert 3)_{1}$  & $ (1,2) $& $ (1,1) $ & $ (0,2) $\\
\\
$(5\vert 0)_{3},$ $(5\vert 0)_{4}$& $ (4,0) $& $ (2,0) $ & $ (3,0) $\\
\\
$(5\vert 0)_{5}$& $ (4,0) $& $ (3,0) $ & $ (2,0) $\\
\\
$(5\vert 0)_{6}$& $ (3,0) $& $ (2,0) $ & $ (3,0) $\\
\\
$(5\vert 0)_{8}$& $ (3,0) $& $ (3,0) $ & $ (2,0) $\\
\\
$(4\vert 1)_{4}$  & $ (2,1) $& $ (2,1) $ & $ (2,0) $\\
\\
$(4\vert 1)_{6}$ & $ (3,1) $& $ (2,1) $ & $ (2,0) $\\
\\ 
$(1\vert 4)_{7}$ & $ (1,3) $& $ (1,1) $ & $ (0,3) $\\
\\
$(1\vert 4)_{8}$ & $ (1,2) $& $ (1,2) $ & $ (0,2) $\\
\\
$(3\vert 2)_{5}$ & $ (0,2) $& $ (0,2) $ & $ (3,0) $\\
\\
$(3\vert 2)_{11},$ $(3\vert 2)_{12}$ & $ (2,1) $& $ (2,1) $ & $ (1,1) $\\
\\
 $(3\vert 2)_{13}$ & $ (2,2) $& $ (1,1) $ & $ (2,1) $\\
\\
$(2\vert 3)_{5},$ $(2\vert 3)_{6},$ $(2\vert 3)_{8},$  $(2\vert 3)_{9},$  $(2\vert 3)_{10},$  $(2\vert 3)_{11}$ & $ (0,3) $& $ (0,3) $ & $ (2,0) $\\
\\
$(2\vert 3)_{13},$  $(2\vert 3)_{14},$   $(2\vert 3)_{16},$  & $ (1,2) $& $ (1,2) $ & $ (1,1) $\\
\\
$(2\vert 3)_{18},$ $(2\vert 3)_{19}$ & $ (2,2) $& $ (0,2) $ & $ (2,1) $\\
\\
$(2\vert 3)_{20}$ & $ (1,2) $& $ (1,1) $ & $ (1,2) $\\
\\
$(2\vert 3)_{22}$ & $ (2,2) $& $ (2,1) $ & $ (0,2) $\\
\\
$(2\vert 3)_{23}$ & $ (1,3) $& $ (1,1) $ & $ (1,2) $\\
\\
$(2\vert 3)_{24}$ & $ (2,2) $& $ (2,1) $ & $ (0,2) $\\
\hline
 \label{ta1}
\end{longtable}
The following lemmas and theorems will be useful in the next section.

\begin{thm}\cite[Theorem 4.2]{Nay1}\label{d1}
Every Heisenberg Lie superalgebra with an even center, having superdimension \( (2m+1 \vert n) \), is isomorphic to \( H(m, n) = H_{\overline{0}} \oplus H_{\overline{1}} \), where 
\begin{equation*}
H_{\overline{0}} = \langle x_1, \dots, x_m, x_{m+1}, \dots, x_{2m}, z \ \vert \ [x_i, x_{m+i}] = z, \ i = 1, \dots, m \rangle,
\end{equation*}
and 
\begin{equation*}
H_{\overline{1}} = \langle y_1, \dots, y_n \ \vert \ [y_j, y_j] = z, \ j = 1, \dots, n \rangle.
\end{equation*}
\end{thm}

\begin{thm}\cite[Page 4]{Nay2}\label{d11}
Every Heisenberg Lie superalgebra with an odd center, having superdimension \( (m \vert m+1) \), is isomorphic to \( H_m = H_{\overline{0}} \oplus H_{\overline{1}} \), where 
\begin{equation*}
H_m = \langle x_1, \dots, x_m, y_1, \dots, y_m, z \ \vert \ [x_j, y_j] = z, \ j = 1, \dots, m \rangle.
\end{equation*}
\end{thm}

We also recall the concept of a capable Lie superalgebra, which will be applied in the next section. A Lie superalgebra \( L \) is called capable if \( L \cong H/Z(H) \) for some Lie superalgebra \( H \).

\begin{thm}\cite[Theorem 5.9]{Rud}\label{l0}
Let \( L \) be a finite-superdimensional Lie superalgebra whose derived subsuperalgebra has dimension one. Then \( L \) is capable if and only if \( L \cong H(1,0) \oplus A(k-3 \vert l) \) or \( L \cong H_1 \oplus A(k-1 \vert l-2) \).
\end{thm}

\begin{prop}\cite[Proposition 3.2]{Rud1}\label{prop0}
Let \( L \) be a \( (p \vert q) \)-generator nilpotent Lie superalgebra with superdimension \( (m \vert n) \) and \( \mathrm{sdim} \ L^2 = (r \vert s) \). Then \( \mathrm{sdim} \ L / L^2 = (p \vert q) \).
\end{prop}

A linear map \( D : L \to L \) of parity \( \alpha \in \mathbb{Z}_2 \) is called a superderivation of \( L \) if, for all \( x, y \in L \),
\begin{equation*}
D[x, y] = [D(x), y] + (-1)^{\alpha \vert x \vert} [x, D(y)].
\end{equation*}
The set of all superderivations of parity \( \alpha \) is denoted \( Der_{\alpha}(L) \), where \( \alpha \in \mathbb{Z}_2 \). The superspace \( Der(L) = Der_{\overline{0}}(L) \oplus Der_{\overline{1}}(L) \) forms a Lie superalgebra with the usual bracket 
\begin{equation*}
[D, E] = DE - (-1)^{\vert D \vert \vert E \vert} ED,
\end{equation*}
where \( D, E \in Der(L) \). The elements of \( Der(L) \) are called superderivations of \( L \). For \( x \in L \), the map \( ad_x : L \to L \) given by \( y \mapsto [x, y] \) is a superderivation of \( L \), referred to as an inner derivation. The set of inner derivations, denoted \( ad(L) \), is a superideal of \( Der(L) \).

A superderivation of a Lie superalgebra \( L \) is called an \( ID \)-superderivation if it maps \( L \) to its derived subsuperalgebra. The set of all \( ID \)-superderivations of \( L \), denoted \( ID(L) \), is defined as \( ID(L) = \{ \alpha \in Der(L) \ \vert \ \alpha(L) \subseteq L^2 \} \), and \( ID^*(L) \) is defined as \( ID^*(L) = \{ \alpha \in ID(L) \ \vert \ \alpha(Z(L)) = 0 \} \). Clearly, \( ad(L) \leq ID^*(L) \leq ID(L) \).

\begin{thm}\cite[Theorem 2.2]{Liu}\label{th1.5}
Suppose \( L \) is a Lie superalgebra such that \( \mathrm{sdim} \ L^2 = (r \vert s) \), with \( r + s \leq \infty \), and \( L / Z(L) \) is finitely generated. Then
\begin{equation*}
\mathrm{sdim} \ ID^* \leq \lambda(L^2, p, q),
\end{equation*}
where \( (p \vert q) \) is the minimal generator pair of \( L / Z(L) \). In particular, \( ID^* \) is finite-superdimensional.
\end{thm}

\section{Main Results}
In this section, we obtain  converse of the Schur's theorem and we show  that $ st(L)=\uplambda (L^{2}, p, q)-\mathrm{sdim} L/Z(L), $ where 
$ \uplambda (L^{2},p,q)=(p.\mathrm{dim} L^{2}_{\overline{0}}+q.\mathrm{dim} L^{2}_{\overline{1}},  q.\mathrm{dim} L^{2}_{\overline{0}}+p.\mathrm{dim} L^{2}_{\overline{1}}) $ by using Theorem  \ref{l21} for $ st(L)=(r_1, s_1) $  such that $ r_1 $ and $ s_1 $ are non-negative integers. 
Finally, we determine  the structure of all finite-superdimensional nilpotent Lie superalgebras $ L $ up to isomorphism when $ st(L)\in \lbrace (0,0), (1,0), (0,1), (2,0), (0,2), (1,1) \rbrace. $
 The next theorem shows that  by putting some extra conditions converse of  the Schur's theorem is true.
\begin{thm}\label{l21}
Let  $ L $ be a Lie superalgebra such that $ \mathrm{sdim}  L^{2}=(r\vert s)$ such that $ r+s \leq \infty  $ and $ L/Z(L) $ be  an $ (p\vert q)$-generator. Then 
\begin{equation*}
\mathrm{sdim} L/Z(L)\leq  \uplambda (L^{2}, p, q). 
\end{equation*} 
\end{thm}
\begin{proof}
Since $ L/Z(L)\cong ad(L)$ and $ ad(L) \leq ID^{*}(L),$  we have  $ \mathrm{sdim} L/Z(L)\leq  \uplambda (L^{2}, p, q) 
 $ by using Theorem \ref{th1.5}.
\end{proof}
We are going to classify the structure of all finite-superdimensional nilpotent Lie superalgebras $ L $ up to isomorphism when $ st(L)\in \lbrace (0,0), (1,0), (0,1), (2,0), (0,2), (1,1) \rbrace. $
 The next  two lemmas  show that it is sufficient to obtain 
the structure of all  nilpotent stem Lie superalgebras $ T $ such that $ \mathrm{sdim} T/Z(T)=\uplambda (T^{2},p,q) -st(T)$ for all $ st(T)\in \lbrace (0,0), (1,0), (0,1), (2,0), (0,2), (1,1) \rbrace. $
\begin{lem}\label{L1}
 Let $L$ be a finite-superdimensional Lie superalgebra. Then $L = T \bigoplus A$  and $Z(T) = L^{2} \cap Z(L),$ where $A$ is abelian and $T$ is non-abelian stem Lie superalgebra.
\end{lem}
\begin{proof}
Put $ M=L^{2}\cap Z(L). $ Since $ Z(L)=M\bigoplus A$ where $ A$ is an abelian Lie superalgebra, we have $ L/L^{2}=T/L^{2}\bigoplus (A+L^{2})/L^{2} $ for some Lie supersubalgebra $ T $ of $ L. $ Hence  $ L=T+A$ and $ T\cap A=0, $  so $ L=T\bigoplus A $ and $ Z(T)=L^{2}\cap Z(L). $ 
\end{proof}
\begin{lem}\label{l1}
Let $ L $ be a finite-superdimensional nilpotent Lie superalgebra. Then   $ L=T \bigoplus A$  where $ T $ is a stem Lie superalgebra and $ A$ is an abelian Lie superalgebra   and  $ st(L)=\uplambda (L^{2},p,q)-\mathrm{sdim} L/Z(L) $ for all $ st(L)=(r_1, s_1)$  such that $r_1+s_1 \geq 0 $ if and only if $ st(T)=\uplambda (T^{2},p,q)-\mathrm{sdim} T/Z(T) $ for all $ st(T)=(r_1, s_1)$  such that $r_1+s_1 \geq 0. $
\end{lem}
\begin{proof}
By using Lemma \ref{L1}, we have $ L=T \bigoplus A.$  Clearly  $ \mathrm{sdim} L/Z(L)=\mathrm{sdim} T/Z(T)$ and so  $ \mathrm{sdim} (L/Z(L))^{2}= \mathrm{sdim} (T/Z(T))^{2}. $  Hence 
$ sd(L/Z(L))=sd(T/Z(T))$ by using Proposition \ref{prop0}. On the other hand, $  \mathrm{sdim} L^{2}= \mathrm{sdim} T^{2}. $ Therefore 
the result follows.
\end{proof}
\begin{prop}\label{prop1}
Let $ L $ be an $ (k\vert l)$-superdimensional nilpotent Lie superalgebra with  $ \mathrm{sdim} L^{2}=(r\vert s)$ such that $r+s\geq 2 $ and the minimal generator number pairs $ (p\vert q)$ of $ L/Z(L). $ 
 Then $ st(L)=(r_1, s_1)$ such that $r_1+s_1> 0. $
\end{prop}
\begin{proof}
By using the proof of Theorem \ref{th1.5},  we have $ ID^{*}(L)=ID_{\overline{0}}^{*}(L) \bigoplus ID_{\overline{1}}^{*}(L),$  
\begin{align*}
&\mathrm{dim} ID_{\overline{0}}^{*}(L) \leq  p.\mathrm{dim} L^{2}_{\overline{0}}+q.\mathrm{dim} L^{2}_{\overline{1}}~~\text{and}\cr
& dim ID_{\overline{1}}^{*}(L) \leq q.\mathrm{dim} L^{2}_{\overline{0}}+p.\mathrm{dim} L^{2}_{\overline{1}}.
\end{align*}
Hence  $st(L)=\uplambda (L^{2},p,q)-\mathrm{sdim} L/Z(L)$ if and only if  $t(L)=\mathrm{dim} L^{2}d(L/Z(L))-\mathrm{dim} L/Z(L),$ where $ d=d(L/Z(L))=p+q $ and $ t(L)=r_1+s_1. $
Let $ L $ be a finite-superdimensional nilpotent Lie superalgebras such that $ \mathrm{dim} L^{2}=r+s\geq 2. $ We claim that $ t(L)=r_1+s_1\geq 1. $  
Let $ L/Z(L) $ be an  abelian Lie superalgebra, then $ \mathrm{dim} L/Z(L)=d $ and so
\begin{align}\label{eq2.1}
\mathrm{dim} L/Z(L)&=d\mathrm{dim} L^{2}-d\mathrm{dim} L^{2}+\mathrm{dim} L/Z(L) \cr
&= d\mathrm{dim} L^{2}+ \mathrm{dim} L/Z(L) (1-\mathrm{dim} L^{2}).
\end{align} 
Since $ \mathrm{dim} L^{2}\geq 2, $   we have $\mathrm{dim} L/Z(L) \leq d\mathrm{dim} L^{2}-1$ by using \eqref{eq2.1} and so $ t(L)\geq 1. $\\
 We know that 
 $$ d((L/Z(L))/Z(L/Z(L)))= d(L/Z_{2}(L))$$
  and so 
\begin{align}\label{eq4}
d((L/Z(L))/Z(L/Z(L)))&=d(L/Z_{2}(L))\cr
&=\mathrm{dim} (L/Z_{2}(L))- \mathrm{dim} (L/Z_{2}(L))^{2}\cr
&=\mathrm{dim} (L/Z_{2}(L))- \mathrm{dim} (L^{2}+Z_2(L))/Z_2(L)\cr
&=\mathrm{dim} (L/Z(L))- \mathrm{dim} (L^{2}+Z_2(L))/Z(L).
\end{align}
 On the other hand,  
 \begin{align*}
 &\mathrm{dim} (L^{2}+Z_2(L))/Z(L)\cr
 &=\mathrm{dim} (L^{2}+Z_2(L))/ (L^{2}+Z(L)) +\mathrm{dim} (L^{2}+Z(L))/Z(L)
   \end{align*}
   thus
 \begin{align*}
d((L/Z(L))/Z(L/Z(L)))&=d- \mathrm{dim} (L^{2}+Z_2(L))/ (L^{2}+Z(L)) 
\end{align*}
by using \eqref{eq4}.
Hence 
\begin{align}\label{e6}
& d(L/Z(L))/Z(L/Z(L)) \mathrm{dim} (L/Z(L))^{2}+\mathrm{dim} Z_2(L)/Z(L)\cr
&=d\mathrm{dim} L^{2}-d \mathrm{dim} L^{2}\cap Z(L)- \mathrm{dim} (L^{2}+Z_2(L))/ (L^{2}+Z(L)) \mathrm{dim}(L/Z(L))^{2}\cr
&+\mathrm{dim} Z_2(L)/Z(L).\cr
\end{align}
Also, $ d=\mathrm{dim} L/(L^{2}+Z(L)) $ and $  \mathrm{dim} L^{2}\cap Z(L)\geq 1, $ hence
\begin{align*}
& d(L/Z(L))/Z(L/Z(L)) \mathrm{dim} (L/Z(L))^{2}+\mathrm{dim} Z_2(L)/Z(L)\cr
&=d\mathrm{dim} L^{2}-d (\mathrm{dim} L^{2}\cap Z(L)-1)-\mathrm{dim} L/(L^{2}+Z(L))\cr
&- \mathrm{dim} (L^{2}+Z_2(L))/ (L^{2}+Z(L)) \mathrm{dim}(L/Z(L))^{2}+\mathrm{dim} Z_2(L)/Z(L).
\end{align*}
Since
\begin{align*}
& \mathrm{dim} Z_2(L)/Z(L) - \mathrm{dim} (L^{2}+Z_2(L))/ (L^{2}+Z(L)) \mathrm{dim}(L/Z(L))^{2}-\mathrm{dim} L/(L^{2}+Z(L))\cr
&\leq  \mathrm{dim} Z_2(L)/Z(L)-\mathrm{dim}(L/Z(L))^{2}-\mathrm{dim} L/(L^{2}+Z(L))\cr
&=\mathrm{dim} Z_2(L)/Z(L)-\mathrm{dim} L/Z(L)\cr
&=-\mathrm{dim} L/Z_2(L),
\end{align*}
we have
\begin{align}\label{eq2.5}
&d(L/Z(L))/Z(L/Z(L)) \mathrm{dim} (L/Z(L))^{2}+\mathrm{dim} Z_2(L)/Z(L)\cr
&\leq d\mathrm{dim} L^{2}-d (\mathrm{dim} L^{2}\cap Z(L)-1)-\mathrm{dim} L/Z_2(L).
\end{align}
Since $\mathrm{dim} (L/Z(L))/Z(L/Z(L))\leq d(L/Z(L))/Z(L/Z(L)) \mathrm{dim} (L/Z(L))^{2}, $ we have 
\begin{equation}\label{eq2.4}
\mathrm{dim} L/Z(L) \leq d\mathrm{dim} L^{2}-d (\mathrm{dim} L^{2}\cap Z(L)-1)-\mathrm{dim} L/Z_2(L)
\end{equation}
by using \eqref{eq2.5}.\\
Assume that $ \mathrm{dim}( L/Z(L))^{2}=1. $ Since $ \mathrm{dim}( L/Z(L))^{2}=1 $  and $ L/Z(L) $ is capable,  $ L/Z(L)$ is isomorphic to  $  H(1,0)\bigoplus A(k-t_1-3\vert l-t_2) $  for all $ k\geq t_1+3 $ and $ l\geq t_2 $ or  $ H_1\bigoplus A(k-t_1-1\vert l-t_2-2)$
 for all $ k\geq t_1+1 $ and $ l\geq t_2+2 $ with $ \mathrm{dim} Z(L)=(t_1, t_2)$ such that $ t_1+t_2\geq 1 $ by using Theorem \ref{l0}.  
 Hence $ \mathrm{dim} Z_2(L)=k+l-2 $ and so $ \mathrm{dim} L/Z_2(L)=2. $
Thus $ t(L)\geq 2 $
 by using \eqref{eq2.4}.
Now, let $ \mathrm{dim} (L/Z(L))^{2}\geq 2. $ We proceed by introduction on $ k+l $ to prove our result. Since $ \mathrm{dim} L^{2}=r+s\geq 2, $  we have $ k+l\geq 4. $  If $ k+l=4, $   then  there is no $ (k\vert l)$-superdimensional nilpotent Lie superalgebra $ L $ such that $ k+l=4 $ and $st(L)=(0,0) $  by looking at the Table \ref{ta1}.
 Since  $ \mathrm{dim} (L/Z(L))^{2}\geq 2, $  we have 
\begin{equation}\label{e3}
 \mathrm{dim} (L/Z(L))/Z(L/Z(L))\leq d(L/Z(L))/Z(L/Z(L)) \mathrm{dim} (L/Z(L))^{2}-1
\end{equation}
by using the induction hypothesis. 
Therefore  $ t(L)\geq 1 $ by using \eqref{e3} and \eqref{eq2.5} and so   $ st(L)=(r_1, s_1)$ such that $r_1+s_1> 0. $
\end{proof}
\begin{thm}\label{th1}
Let $ L $ be an $ (k\vert l)$-superdimensional nilpotent Lie superalgebra such that  $ L/Z(L) $ be finitely generated and  $ st(L)=(0,0).$ Then $ L $ is isomorphic to  one of the Lie superalgebra $ A(k\vert l) $ for all $ k+l\geq 1,$ $  H(m, n)\bigoplus A(k-2m-1\vert l-n) $ for all $ k\geq 2m+1 $ and $ l\geq n $ or  $  H_m\bigoplus A(k-m\vert l-m-1)$ for all $ k\geq m $ and $ l\geq m+1. $
\end{thm}
\begin{proof}
Let $ \mathrm{sdim} L^{2}=(r\vert s).$ If $r+s\geq 2, $ then there is no such a nilpotent Lie superalgebra $ L $ by using Proposition \ref{prop1}. Hence $ r+s\leq 1 $ and so $ L $ is isomorphic to an abelian Lie superalgebra $ A(k\vert l) $ for all $ k+l\geq 1,$ $  H(m, n)\bigoplus A(k-2m-1\vert l-n) $ for all $ k\geq 2m+1 $ and $ l\geq n $ or  $  H_m\bigoplus A(k-m\vert l-m-1)$ for all $ k\geq m $ and $ l\geq m+1$ by using Lemma \ref{l1}, Theorems  \ref{d1} and \ref{d11}.
\end{proof}
\begin{prop}\label{prop2}
Let $ L $ be an $ (k\vert l)$-superdimensional nilpotent Lie superalgebra with $ \mathrm{sdim} L^{2}=(r\vert s)$ such that $r+s\geq 3 $ and 
the minimal generator number pairs $ (p\vert q)$ of $ L/Z(L). $ 
   Then $ st(L)=(r_1,s_1)$ such that $r_1+s_1\geq 2. $
\end{prop}
\begin{proof}
Similar to the proof of Proposition \ref{prop1}, we have $st(L)=\uplambda (L^{2},p,q)-\mathrm{sdim} L/Z(L)$ if and only if  $t(L)=\mathrm{dim} L^{2}d(L/Z(L))-\mathrm{dim} L/Z(L)$ where $ t(L)=r_1+s_1. $
Let $ \mathrm{dim} (L/Z(L))^{2}=0. $ Since $ r+s\geq 3, $  we have $ t(L)=r_1+s_1 \geq 2$ by using \eqref{eq2.1}.
If $ \mathrm{dim} (L/Z(L))^{2}=1, $ then $\mathrm{dim} L/Z_2(L)=2 $ by using a similar method used in the proof of Proposition \ref{prop1}. So 
$ t(L)\geq 2 $  by using \eqref{eq2.4}.
Now, let $ \mathrm{dim} (L/Z(L))^{2}=2. $ Then
\begin{align}\label{eq2.6}
 \mathrm{dim} (L/Z(L))/Z(L/Z(L))\leq d(L/Z(L))/Z(L/Z(L)) \mathrm{dim} (L/Z(L))^{2}-1
\end{align}
 by using Proposition \ref{prop1}. Also, $ \mathrm{dim} (L^{2}\cap Z(L))-1\geq 2 $ or $ \mathrm{dim} L/Z_2(L)\geq 1 $ thus   
 $ t(L)=r_1+s_1\geq 2$ by using \eqref{eq2.6} and \eqref{eq2.5}.
Let $ \mathrm{dim} (L/Z(L))^{2}\geq 3. $ Since $ r+s\geq 3, $ we have $ k+l\geq 5. $ If $k+l=5, $ then there is no such a Lie superalgebra such that $ r_1+s_1\leq 1 $ by looking at Table \ref{ta1} and Lemma \ref{l1}. Since $ \mathrm{dim} (L/Z(L))^{2}\geq 3, $ 
\begin{equation}\label{le}
 \mathrm{dim} (L/Z(L))/Z(L/Z(L))\leq d(L/Z(L))/Z(L/Z(L)) \mathrm{dim} (L/Z(L))^{2}-2.
\end{equation}
by using the induction hypothesis. 
  One can see  $ t(L)=r_1+s_1\geq 2 $ by using  \eqref{le} and \eqref{eq2.5}, so $ st(L)=(r_1,s_1)$ such that $r_1+s_1\geq 2. $
\end{proof}
\begin{thm}\label{th2}
Let $ L $ be an $ (k\vert l)$-superdimensional nilpotent Lie superalgebra and   $ L/Z(L) $ be finitely generated with   $ st(L)=(r_1,s_1)$ such that $ r_1+s_1=1. $  
\begin{itemize}
\item[(i).] If $ st(L)=(1,0), $ then $ L $ is isomorphic to  $ (4\vert 0)_{2} \bigoplus A(k-4\vert l) $ for all $ k\geq 4 $ and $ l\geq 0 $ or   $ (1\vert 3)_{1} \bigoplus A(k-1\vert l-3) $ for all $ k\geq 1$ and $ l\geq 3. $
\item[(ii).] If $ st(L)=(0,1), $ then there is no such a  Lie superalgebra.
\end{itemize}
\end{thm}
\begin{proof}
Assume that  $ \mathrm{sdim} L^{2}=(r\vert s). $ If $ r+s\leq 1, $ then $ L $ is isomorphic to $ A(k\vert l) $  for all $ k+l\geq 1, $ $  H(m, n)\bigoplus A(k-2m-1\vert l-n) $ for all $ k\geq 2m+1 $ and $ l\geq n $ or  $  H_m\bigoplus A(k-m\vert l-m-1)$ for all $ k\geq m $ and $ l\geq m+1 $ by using Lemma \ref{l1} and Theorems  \ref{d1} and \ref{d11},
so $ st(L)=(0,0). $  Hence there is no such a nilpotent Lie superalgebra with  $  st(L)=(r_1,s_1)$ such that $ r_1+s_1= 1 $ when  $ r+s\leq 1. $
Assume that  $r+s \geq 3. $ Since  $r+s \geq 3, $
 there is no such a nilpotent Lie superalgebra $ L $  with  $  st(L)=(r_1, s_1)$ such that $ r_1+ s_1=1 $ by using Proposition \ref{prop2}. Let $ r+s=2. $ Then $ L $ is of nilpotency  two or three. In the rest it is sufficient to obtain the structure of all stem nilpotent Lie superalgebras  $ L $ with $ st(L)= (r_1, s_1) $ such that $ r_1+s_1=1 $  and then  we determine   the structure of  all nilpotent Lie superalgebras  $ L $ with $ st(L)= (r_1, s_1) $ such that $ r_1+s_1=1 $    by using Lemma \ref{l1}. Assume that $ L $ is a stem Lie superalgebra  of nilpotency class two.  Hence $ Z(L)=L^{2} $ and so   $ d(L/Z(L))=\mathrm{dim} L/Z(L). $ On the other hand, $ st(L)= (r_1, s_1) $ such that $ r_1+s_1=1 $  by our assumption. Thus $ k+l=3. $ But there is no $ (k\vert l)$-superdimensional stem  nilpotent Lie superalgebra $ L $ such that $ k+l=3 $ and $ st(L)= (r_1, s_1) $ such that $ r_1+s_1=1 $  by looking at the classification of all nilpotent Lie superalgebras in \cite{Al}.  Let $ L $ be a stem Lie superalgebra of nilpotency class $ 3. $ Then $ \mathrm{dim} Z(L)=(t_1, t_2) $ such that $ t_1+t_2=1.$ So,
 $ d(L/Z(L))=k+l-2 $ and $ \mathrm{dim} L/Z(L)=k+l-1. $ Since  $  r_1+s_1=1 $ by our assumption, we have $ l+k=4. $ By looking at Table \ref{ta1}, we have $ L $ is isomorphic to $ (4\vert 0)_{2} $ or $ (1\vert 3)_{1} $
  when $ L $ is a stem nilpotent Lie superalgebra. Hence  $ L $ is isomorphic to 
 $ (4\vert 0)_{2} \bigoplus A(k-4\vert l) $ for all $ k\geq 4 $ and $ l\geq 0 $ or   $ (1\vert 3)_{1} \bigoplus A(k-1\vert l-3) $ for all $ k\geq 1$ and $ l\geq 3$  by using Lemma \ref{l1}.   Therefore if $ st(L)=(1,0), $ then $ (4\vert 0)_{2} \bigoplus A(k-4\vert l) $ for all $ k\geq 4 $ and $ l\geq 0 $ or   $ (1\vert 3)_{1} \bigoplus A(k-1\vert l-3) $ for all $ k\geq 1$ and $ l\geq 3$  and there is no such a Lie superalgebra with $ st(L)=(0,1). $
\end{proof}
\begin{lem}\label{l27}
Let $ L $ be a stem nilpotent $ (k\vert l)$-superdimensional Lie superalgebra with $ \mathrm{sdim}L^{2}=(r\vert s) $ such that $ k+l=6 $ and $ r+s\geq 4. $ Then $ st(L)= (r_1, s_1) $ such that $ r_1+ s_1\geq 3. $
\end{lem}
\begin{proof}
On the contrary, let  $ st(L)= (r_1, s_1) $ such that $ r_1+ s_1\leq 2. $ Then similar to the proof of Proposition \ref{prop1} we have $st(L)=\uplambda (L^{2},p,q)-\mathrm{sdim} L/Z(L)$ if and only if  $t(L)=\mathrm{dim} L^{2}d(L/Z(L))-\mathrm{dim} L/Z(L)$ where 
$ t(L)=r_1+s_1 $ and $ d(L/Z(L))=p+q. $
  Thus we can consider $ t(L)=r_1+s_1\leq 2. $ Since $ k+l=6 $ and  $ r+s\geq 4, $ we have $ \mathrm{dim} L^{2} $ is equal to $ 4 $ or $ 5. $
 If $ t(L)=r_1+s_1 $ is equal to $ 0 $ or $ 1, $ then we get a contradiction by using Theorems \ref{th1} and \ref{th2}.
If $ t(L)=r_1+s_1=2 $ and $ \mathrm{dim} L^{2}=4, $ then $ d(L/Z(L))=2. $ Hence 
$\mathrm{dim} L/Z(L)=6  $ and so $ \mathrm{dim} Z(L)=0. $ It is a contradiction. If  $ t(L)=r_1+s_1=2 $ and $ \mathrm{dim} L^{2}=5, $ then $ d(L/Z(L))=1. $
Thus $ \mathrm{dim} L/Z(L)=3 $ and $ \mathrm{dim} (L/Z(L))^{2}=2. $ Now, by looking out  \cite{Al} there is no such a Lie superalgebra       $ L/Z(L) $ which  is a contradiction.
Therefore $ st(L)= (r_1, s_1) $ such that $ r_1+ s_1\geq 3. $
\end{proof}
\begin{prop}\label{prop3}
Let $ L $ be   an  $ (k\vert l)$-superdimensional nilpotent Lie superalgebra with $ \mathrm{sdim} L^{2}=(r\vert s)$ such that $r+s\geq 4 $
and the minimal generator number pairs $ (p\vert q)$ of $ L/Z(L). $ 
 Then $ st(L)=(r_1,s_1)$ such that $r_1+s_1\geq 3. $
\end{prop}
\begin{proof}
By using a similar way used in the proof of Theorem \ref{prop1}, we have $st(L)=\uplambda (L^{2},p,q)-\mathrm{sdim} L/Z(L)$ if and only if  $t(L)=\mathrm{dim} L^{2}d(L/Z(L))-\mathrm{dim} L/Z(L),$ where $ t(L)=r_1+s_1 $ and $ d(L/Z(L))=p+q. $
Now, let $ \mathrm{dim} (L/Z(L))^{2}=0. $ Since $ r+s \geq 4, $  we have $ t(L) \geq 3$ by using \eqref{eq2.1}.
Assume that  $ \mathrm{dim} (L/Z(L))^{2}=1. $ Since $ L/Z(L) $ is capable,  we have 
$ L/Z(L)$ is isomorphic to  $  H(1,0)\bigoplus A(k-t_1-3\vert l-t_2) $  for all $ k\geq t_1+3 $ and $ l\geq t_2 $ or  $ H_1\bigoplus A(k-t_1-1\vert l-t_2-2)$
 for all $ k\geq t_1+1 $ and $ l\geq t_2+2 $ with $ \mathrm{sdim} Z(L)=(t_1, t_2)$ such that $ t_1+t_2\geq 1 $ by using Lemma \ref{l0}.
 On the other hand, $ \mathrm{dim} Z_2(L)=k+l-2 $ and so $ \mathrm{dim} L/Z_2(L)=2. $  Also, $ \mathrm{dim} L^{2}/L^{2} \cap Z(L)=1 $ and since $ \mathrm{dim} L/Z_2(L)=2, $ we have 
$ \mathrm{dim} L/Z(L)/Z(L/Z(L))\leq d(L/Z_2(L))\mathrm{dim} (L/Z(L))^{2}, $  we have $ t(L)\geq 3 $ by using \eqref{eq2.5}.\\
Assume that $ \mathrm{dim} (L/Z(L))^{2}= 2.$  By using  Proposition \ref{prop1}  we have
\begin{equation}\label{e4}
\mathrm{dim} L/Z(L)/Z(L/Z(L))\leq d(L/Z_2(L))\mathrm{dim} (L/Z(L))^{2}-1.
\end{equation}
If $ t(L/Z(L))=1, $ then  $ L/Z(L)$ is isomorphic to 
 $ (4\vert 0)_{2} \bigoplus A(k-t_1-4\vert l-t_2) $ for all $ k-t_1\geq 4 $ and $ l-t_2\geq 0 $ or   $  (1\vert 3)_{1} \bigoplus A(k-t_1-1\vert l-t_2-3) $ for all $ k-t_1\geq 1$ and $ l-t_2\geq 3 $ with $ \mathrm{sdim} Z(L)=(t_1, t_2)$ such that $ t_1+t_2\geq 1 $ and  $ r+s\geq 4 $   by using Theorem \ref{th2}. 
 Hence  $ \mathrm{dim} L^{2}\cap Z(L)= 2$ and  $ \mathrm{dim} L/Z_2(L)\geq 1. $ 
 Therefore $ t(L)\geq 3$  by using \eqref{e4} and \eqref{eq2.5}. If $  t(L/Z(L))\geq 2 $ , then
 \begin{equation}\label{eq2.9}
\mathrm{dim} L/Z(L)/Z(L/Z(L))\leq d(L/Z_2(L))\mathrm{dim} (L/Z(L))^{2}-2
\end{equation}
and so 
\begin{align}\label{eq2.10}
\mathrm{dim} L/Z(L) \leq d\mathrm{dim} L^{2}-d (\mathrm{dim} L^{2}\cap Z(L)-1)-\mathrm{dim} L/Z_2(L)-2
\end{align}
by using \eqref{eq2.5}.
 Since $ (\mathrm{dim} (L^{2}\cap Z(L))-1)\geq 3 $ or $ \mathrm{dim} L/Z_2(L)\geq 1, $  we have $  t(L)=r_1+s_1\geq 3. $
  If $ \mathrm{dim} (L/Z(L))^{2}= 3,$ then
 \begin{equation}\label{eq2.11}
 \mathrm{dim} (L/Z(L))/Z(L/Z(L))\leq d(L/Z(L))/Z(L/Z(L)) \mathrm{dim} (L/Z(L))^{2}-2
\end{equation}
by using Proposition \ref{prop2}. On the other hand, $ \mathrm{dim} (L^{2}\cap Z(L))-1\geq 3 $ or $ \mathrm{dim} L/Z_2(L)\geq 1 $ thus $  t(L)=r_1+s_1\geq 3 $ by using \eqref{eq2.11} and \eqref{eq2.5}.
  \\
Let  $ \mathrm{dim} (L/Z(L))^{2}\geq  4.$ Then we claim that $  t(L)=r_1+s_1\geq 3. $ 
Since   $ r+s\geq 4, $ we have $ \mathrm{dim} L=l+k\geq 6. $ If $ l+k=6, $ then $  t(L)=r_1+s_1\geq 3 $ by using Lemma \ref{l27}.
Since $ \mathrm{dim} (L/Z(L))^{2}=4, $
\begin{equation}\label{eq2.12}
 \mathrm{dim} (L/Z(L))/Z(L/Z(L))\leq d(L/Z(L))/Z(L/Z(L)) \mathrm{dim} (L/Z(L))^{2}-3
\end{equation}
by using the induction hypothesis.
 Also,  $ (\mathrm{dim} (L^{2}\cap Z(L))-1)\geq 3 $ or $ \mathrm{dim} L/Z_2(L)\geq 1. $ One can see  $  t(L)=r_1+s_1\geq 3 $
   by using \eqref{eq2.12} and \eqref{eq2.5}. Therefore $ st(L)=(r_1,s_1)$ such that $r_1+s_1\geq 3. $
  \end{proof}
\begin{thm}\label{th3}
Let $ L $ be an $ (k\vert l)$-superdimensional nilpotent Lie algebra and   $ L/Z(L) $ be finitely generated and   $ st(L)=(r_1,s_1)$ such that $r_1+s_1= 2. $ Then
\begin{itemize}
\item[(i).] if $ st(L)=(2,0), $ then $ L $ is isomorphic to one of the Lie superalgebras  $ (5\vert 0)_{3}\bigoplus A(k-5,l) $ for all $ k\geq 5 $ and $ l\geq 0, $ $ (5\vert 0)_{4}\bigoplus A(k-5,l) $ for all $ k\geq 5 $ and $ l\geq 0, $
 $ (5\vert 0)_{5}\bigoplus A(k-5,l) $ for all $ k\geq 5 $ and $ l\geq 0, $
  $ (1\vert 4)_{7}\bigoplus A(k-1, l-4)$ for all $ k\geq 1 $ and $ l\geq 4 $
  or $ (2\vert 3)_{23}\bigoplus A(k-2,l-3) $   for all $ k\geq 2 $ and $ l\geq 3. $  
\item[(ii).] If $ st(L)=(0,2), $ then $ L $ is isomorphic to one of the Lie superalgebras $ (2\vert 2)_{1}\bigoplus A(k-2,l-2) $  for all $ k\geq 2 $ and $ l\geq 2, $
$ (2 \vert 2)_{4}\bigoplus A(k-2,l-2)$ for all $ k\geq 2 $ and $ l\geq 2, $
  $ (2\vert 3)_{18}\bigoplus A(k-2,l-3) $  for all $ k\geq 2 $ and $ l\geq 3, $ $ (2\vert 3)_{19}\bigoplus A(k-2,l-3) $ for all $ k\geq 2 $ and $ l\geq 3, $ or $ (2\vert 3)_{24}\bigoplus A(k-2,l-3) $ for all $ k\geq 2 $ and $ l\geq 3. $ 
\item[(iii).] If $ st(L)=(1,1), $ then $ L $ is isomorphic to one of the Lie superalgebras $ (2 \vert 2)_{6}\bigoplus A(k-2, l-2)$ for all $ k\geq 2 $ and $ l\geq 2, $
$ (4\vert 1)_{6}\bigoplus A(k-4, l-1)$   for all $ k\geq 4$ and $ l\geq 1 $ or $ (3\vert 2)_{13}\bigoplus A(k-3,l-2) $ for all $ k\geq 3 $ and $ l\geq 2, $
\end{itemize}
\end{thm}
\begin{proof}
Let $ \mathrm{sdim} L^{2}=(r, s). $ If $ r+ s\geq 4, $ then there is no such a nilpotent Lie superalgebra by using Proposition \ref{prop3}. 
Hence $  r+ s\leq 3. $ If $ r+s\leq 1, $ then $ L $ is isomorphic to $ A(k\vert l) $  for all $ k+l\geq 1, $ $  H(m, n)\bigoplus A(k-2m-1\vert l-n) $ for all $ k\geq 2m+1 $ and $ l\geq n $ or  $  H_m\bigoplus A(k-m\vert l-m-1)$ for all $ k\geq m $ and $ l\geq m+1 $ by using Lemma \ref{l1} and Theorems  \ref{d1} and \ref{d11}, so $ st(L)=(0,0) .$ Hence in this case there is no such a nilpotent Lie superalgebras $ L $ with  $ r+s=1 $ and  $ st(L)=(r_1,s_1)$ such that $r_1+s_1= 2. $ 
 Let $ r+ s=2. $   It is sufficient to obtain the structure of all  stem nilpotent Lie superalgebras $ L $  with $ st(L)=(r_1,s_1)$ such that $r_1+s_1=2 $ and then we can determine   the structure of  all nilpotent Lie algebras $ L $ with $ r+s=2 $ and $ st(L)=(r_1,s_1)$ such that $r_1+s_1= 2 $   by using Lemma \ref{l1}. 
Let $ L $ is of nilpotency class two and stem. Since $ r+s=2, $ $ st(L)=(r_1,s_1)$ such that $r_1+s_1= 2 $  and $ L $ is of nilpotency class two, we have $ Z(L)=L^{2} $ and so $ k+l=4. $
By looking at the Table \ref{ta1}  $ L $ is isomorphic to one of the Lie superalgebras $ (2\vert 2)_{1}, $ $ (2\vert 2)_{4} $ or $ (2\vert 2)_{6}. $
 Let  $ L $ be of nilpotency class three and stem. Then $ \mathrm{sdim} Z(L)=(m,n) $ such that $ m+n=1. $
Since  $ r+s=2, $ $ st(L)=(r_1,s_1)$ such that $r_1+s_1= 2 $  and $ \mathrm{sdim} Z(L)=(m,n) $ such that $ m+n=1, $ we have  $ k+l=5 .$
Thus $ L $ is isomorphic to one of the Lie superalgebras $ (5\vert 0)_{5}, $    $ (2\vert 3)_{24}, $ or $ (4\vert 1)_{6} $  by looking  Table \ref{ta1}.
In the case $ r+ s=3 $ and $ L $ is a stem nilpotent Lie superalgebra $ L, $  we can see that $ k+l $ is equal to $ 4 $ or $ 5 $ by using a similar method and so $ L $ is isomorphic to one of the Lie superalgebras $ (5\vert 0)_{3}, $ $ (5\vert 0)_{4}, $ $ (1\vert 4)_{7}, $ $ (3\vert 2)_{13}, $   $ (2\vert 3)_{18}, $ $ (2\vert 3)_{19}, $ or $ (2\vert 3)_{23} $
 by using at Table \ref{ta1}. Now, if $ st(L)=(2,0), $ then  $ L $ is isomorphic to one of the Lie superalgebras  $ (5\vert 0)_{3}\bigoplus A(k-5,l) $ for all $ k\geq 5 $ and $ l\geq 0, $ $ (5\vert 0)_{4}\bigoplus A(k-5,l) $ for all $ k\geq 5 $ and $ l\geq 0, $
 $ (5\vert 0)_{5}\bigoplus A(k-5,l) $ for all $ k\geq 5 $ and $ l\geq 0, $
  $ (1\vert 4)_{7}\bigoplus A(k-1, l-4)$ for all $ k\geq 1 $ and $ l\geq 4, $
  or $ (2\vert 3)_{23}\bigoplus A(k-2,l-3) $   for all $ k\geq 2 $ and $ l\geq 3 $    by using Lemma \ref{l1}. 
 If $ st(L)=(0,2), $ then          $ L $ is isomorphic to one of the Lie superalgebras $ (2\vert 2)_{1}\bigoplus A(k-2,l-2) $  for all $ k\geq 2 $ and $ l\geq 2, $
$ (2 \vert 2)_{4}\bigoplus A(k-2,l-2)$ for all $ k\geq 2 $ and $ l\geq 2, $
  $ (2\vert 3)_{18}\bigoplus A(k-2,l-3) $  for all $ k\geq ,2 $
   $ (2\vert 3)_{19}\bigoplus A(k-2,l-3) $  for all $ k\geq 2, $
 or
   and $ l\geq 3. $ $ (2\vert 3)_{24}\bigoplus A(k-2,l-3) $ for all $ k\geq 2 $ and $ l\geq 3 $ by using Lemma \ref{l1}. 
 If $ st(L)=(1,1), $ then        $ L $ is isomorphic to $ (2 \vert 2)_{6}\bigoplus A(k-2, l-2)$ for all $ k\geq 2 $ and $ l\geq 2, $
$ (4\vert 1)_{6}\bigoplus A(k-4, l-1)$   for all $ k\geq 4$ and $ l\geq 1 $ or $ (3\vert 2)_{13}\bigoplus A(k-3,l-2) $ for all $ k\geq 3 $ and $ l\geq 2$   by using Lemma \ref{l1}. 
\end{proof}
\begin{theorem}
Let $ L $ be an  $ (k\vert l)$-superdimensional nilpotent Lie superalgebra such that  $ L/Z(L) $ be finitely generated.  Then 
\begin{itemize}
\item[(i).] $ st(L)=(0,0)$  if and only if  $ L $ is isomorphic to one of Lie superalgebras $ A(k\vert l) $ for all $ k+l\geq 1,$ $  H(m, n)\bigoplus A(k-2m-1\vert l-n) $ for all $ k\geq 2m+1 $ and $ l\geq n $ or  $  H_m\bigoplus A(k-m\vert l-m-1)$ for all $ k\geq m $ and $ l\geq m+1. $
\item[(ii).] $ st(L)=(1,0)$ if and only if $ L $ is isomorphic to $ (4\vert 0)_{2} \bigoplus A(k-4\vert l) $ for all $ k\geq 4 $ and $ l\geq 0 $ or   $ (1\vert 3)_{1} \bigoplus A(k-1\vert l-3) $ for all $ k\geq 1$ and $ l\geq 3. $
\item[(iii).] 
There is no nilpotent  Lie superalgebra with $ st(L)=(0,1).$
\item[(iv).]$ st(L)=(2,0) $ if and only if  $ L $ is isomorphic to one of the Lie superalgebras  $ (5\vert 0)_{3}\bigoplus A(k-5,l) $ for all $ k\geq 5 $ and $ l\geq 0, $ $ (5\vert 0)_{4}\bigoplus A(k-5,l) $ for all $ k\geq 5 $ and $ l\geq 0, $
 $ (5\vert 0)_{5}\bigoplus A(k-5,l) $ for all $ k\geq 5 $ and $ l\geq 0, $
  $ (1\vert 4)_{7}\bigoplus A(k-1, l-4)$ for all $ k\geq 1 $ and $ l\geq 4 $
  or $ (2\vert 3)_{23}\bigoplus A(k-2,l-3) $   for all $ k\geq 2 $ and $ l\geq 3. $  
\item[(v).]$ st(L)=(0,2) $ if and only if  $ L $ is isomorphic to one of the Lie superalgebras $ (2\vert 2)_{1}\bigoplus A(k-2,l-2) $  for all $ k\geq 2 $ and $ l\geq 2, $
$ (2 \vert 2)_{4}\bigoplus A(k-2,l-2)$ for all $ k\geq 2 $ and $ l\geq 2, $
  $ (2\vert 3)_{18}\bigoplus A(k-2,l-3) $  for all $ k\geq 2 $ and $ l\geq 3, $ $ (2\vert 3)_{19}\bigoplus A(k-2,l-3) $ for all $ k\geq 2 $ and $ l\geq 3, $ or $ (2\vert 3)_{24}\bigoplus A(k-2,l-3) $ for all $ k\geq 2 $ and $ l\geq 3. $ 
\item[(vi).] $ st(L)=(1,1) $  if and only if $ L $ is isomorphic to $ (2 \vert 2)_{6}\bigoplus A(k-2, l-2)$ for all $ k\geq 2 $ and $ l\geq 2, $
$ (4\vert 1)_{6}\bigoplus A(k-4, l-1)$   for all $ k\geq 4$ and $ l\geq 1 $ or $ (3\vert 2)_{13}\bigoplus A(k-3,l-2) $ for all $ k\geq 3 $ and $ l\geq 2. $
\end{itemize}
\end{theorem} 
\begin{proof}
The  results follow by using Theorems \ref{th1}, \ref{th2} and \ref{th3}.
  The converse of theorem is obvious.
\end{proof}
\begin{cor} 
Let $ L $ be a finite-superdimensional nilpotent Lie superalgebra and  $ L/Z(L) $ be finitely generated  with    $ \mathrm{sdim} L^{2}=(r\vert s)$ such that $r+s \geq 4. $  Then $ st(L)=(r_1, s_1)$ such that $r_1+s_1\geq 3. $
\end{cor}
In the rest we show that there exists at least a finite-superdimensional nilpotent Lie superalgebra $ L $ such that $ st(L)=\uplambda (L^{2},p,q) -\mathrm{sdim} L/Z(L)$ for $ st(L)=(r_1, s_1) $  such that $ r_1 $ and $ s_1 $ are non-negative integers and $ r_1+s_1\geq 0. $ 
\begin{thm}
There exists at least a finite-superdimensional  nilpotent Lie superalgebra   $ L $ for $ st(L)=(r_1, s_1) $  such that $ r_1 $ and $ s_1 $ are non-negative integers $ r_1+s_1\geq 0 $  with  $ L/Z(L) $ be finitely generated.  
\end{thm}
\begin{proof}
If $ st(L)=(0,0), $ then there is such a Lie superalgebra by using Theorem \ref{th1}.  We claim that $ L\cong \langle s, s_j\mid [s, s_i]=s_{i+1}, 1\leq j\leq t+2, 1\leq i\leq t+1 \rangle $ satisfies in our assumption for all $ st(L)= (r\vert s) $ such that $ r+s\geq 1. $ We proceed by induction on $ st(L). $ If $ st(L)=(r\vert s) $ such that $ r+s=1, $ then
\begin{equation*}
 L=(4\vert 0)_{2}\cong  \langle s, s_j\mid [s, s_i]=s_{i+1}, 1\leq j\leq 3, 1\leq i\leq 2 \rangle
\end{equation*}
 by using Theorem \ref{th2}. By using the induction hypothesis, let 
 \begin{equation*}
 L\cong (t+2\vert 0) =\langle s, s_j\mid [s, s_i]=s_{i+1},  s, s_j\in 1\leq j\leq t+1, 1\leq i\leq t \rangle
 \end{equation*}
 satisfies  in our hypothesis.
 Put 
  \begin{equation*}
  H/Z(H)= \langle s+Z(H), s_j+Z(H)\mid [s_i,s]+Z(H)=s_{i+1}+Z(H), 1\leq j\leq t+1, 1\leq i\leq t \rangle.
  \end{equation*} 
 Hence 
$ \varphi : H/Z(H) \longrightarrow L=(t+2\vert 0)=\langle s, s_j\mid [s_i,s]=s_{i+1}, 1\leq j\leq t+1, 1\leq i\leq t \rangle$ by given
\begin{align*}
& s+Z(H)\mapsto s,\cr
& s_j+Z(H)\mapsto s_j
\end{align*} 
 is an isomorphism.
Hence
 \begin{align*}
 H\cong (t+3\vert 0)=\langle s, s_j\mid [s,s_i]=s_{i+1}, 1\leq j\leq t+2, 1\leq i\leq t+1  \rangle
 \end{align*}
and $ Z(H)=\langle s_{t+2} \rangle.  $ Clearly   $ st(L)=\uplambda (L^{2},p,q)-\mathrm{sdim} L/Z(L),$ as required. 
\end{proof}

\end{document}